\documentclass[11pt]{amsart}
\usepackage{amsfonts,amssymb,amsmath,amsthm}
\usepackage{url}
\usepackage{enumerate}

\newtheorem{thrm}{Theorem}[section]
\newtheorem{lem}[thrm]{Lemma}

\theoremstyle{definition}

\numberwithin{equation}{section}

\author{Wen-ming Lu and Lin Zhang}
\address{
Wen-ming Lu\newline\indent School of Science\\Hangzhou Dianzi
University\\Hangzhou\\People's Republic of
China}\email{lu\_wenming@163.com}
\address{
Lin Zhang\newline\indent Department of Mathematics\\
Zhejiang University\\Hangzhou\\People's Republic of China}
\email{godyalin@163.com;linyz@zju.edu.cn}

\keywords{Combinations of Bernstein polynomials; Functions with
endpoint singularities; Direct and inverse results}
\subjclass{Primary 41A10, Secondary 41A17}
\begin{document}

\title[Pointwise weighted approximation of functions with endpoint singularities]{Pointwise weighted approximation of functions with endpoint singularities by combinations of Bernstein operators}

\begin{abstract}
We give direct and inverse theorems  for the weighted approximation
of functions with endpoint singularities by combinations of
Bernstein operators.
\end{abstract}
\maketitle

\section{Introduction}
The set of all continuous functions, defined on the interval $I$, is
denoted by $C(I)$. For any $f\in C([0,1])$, the corresponding
\emph{Bernstein operators} are defined as follows:
$$B_n(f,x):=\sum_{k=0}^nf(\frac{k}{n})p_{n,k}(x),$$
where \begin{align*}p_{n,k}(x):={n \choose k}x^k(1-x)^{n-k}, \
k=0,1,2,\ldots,n, \ x\in[0,1].\end{align*}
Approximation properties of Bernstein operators have been studied
very well (see \cite{Berens}, \cite{Della},
\cite{Totik}-\cite{Lorentz}, \cite{Yu}-\cite{X. Zhou}, for example).
In order to approximate the functions with singularities, Della
Vecchia et al. \cite{Della} and Yu-Zhao \cite{Yu} introduced some
kinds of \emph{modified Bernstein operators}. Throughout the paper,
$C$ denotes a positive constant independent of $n$ and $x$, which
may be different in different cases.\\
Let
\begin{align*}
w(x)=x^\alpha(1-x)^\beta,\ \alpha,\ \beta \geqslant0,\ \alpha +
\beta >0,\ 0 \leqslant x \leqslant 1.
\end{align*}
and
\begin{align*}
C_w:= \{{f \in C((0,1)) :\lim\limits_{x\longrightarrow
1}(wf)(x)=\lim\limits_{x\longrightarrow 0}(wf)(x)=0 }\}.
\end{align*}

The \emph{norm} in $C_w$ is defined by
$\|wf\|_{C_w}:=\|wf\|=\sup\limits_{0\leqslant x\leqslant
1}|(wf)(x)|$. Define
\begin{align*}
W_{w,\lambda}^{r}:= \{f \in C_w:f^{(r-1)} \in A.C.((0,1)),\
\|w\varphi^{r\lambda}f^{(r)}\|<\infty\}.
\end{align*}

For $f \in C_w$, define the \emph{weighted modulus of smoothness} by
\begin{align*}
\omega_{\varphi^\lambda}^{r}(f,t)_w:=\sup_{0<h\leqslant
t}\{\|w\triangle_{h\varphi^\lambda}^{r}f\|_{[16h^2,1-16h^2]}+\|w\overrightarrow{\triangle}_{h}^{r}f\|_{[0,16h^2]}+\|w\overleftarrow{\triangle}_{h}^{r}f\|_{[1-16h^2,1]}\},
\end{align*}
where
\begin{align*}
\Delta_{h\varphi}^{r}f(x)=\sum_{k=0}^{r}(-1)^{k}{r \choose
k}f(x+(\frac r2-k)h\varphi(x)),\\
\overrightarrow{\Delta}_{h}^{r}f(x)=\sum_{k=0}^{r}(-1)^{k}{r
\choose k}f(x+(r-k)h),\\
\overleftarrow{\Delta}_{h}^{r}f(x)=\sum_{k=0}^{r}(-1)^{k}{r \choose
k}f(x-kh),
\end{align*}
and $\varphi(x)=\sqrt{x(1-x)}$. Della Vecchia \emph{et al.} firstly
introduced $B_{n}^{\ast}(f,x)$ and ${\bar{B}}_{n}(f,x)$ in
\cite{Della}, where the properties of $B_{n}^{\ast}(f,x)$ and
${\bar{B}}_{n}(f,x)$ are studied. Among others, they prove that
\begin{align*}
\|w(f-B_{n}^{\ast}(f))\|\leqslant
C\omega_{\varphi}^{2}(f,n^{-1/2}),\ f\in C_{w},\\
\|{\bar{w}}(f-{\bar{B}_{n}(f)})\|\leqslant
\frac{C}{n^{3/2}}\sum_{k=1}^{[\sqrt{n}]}k^{2}\omega_{\varphi}^{2}(f,\frac{1}{k})_{\bar{w}}^{\ast},\
f\in C_{\bar{w}},
\end{align*}
where $w(x)=x^{\alpha}(1-x)^{\beta},\ \alpha,\ \beta\geqslant 0,\
\alpha+\beta>0,\ 0\leqslant x \leqslant1.$ In \cite{S. Yu}, for any
$\alpha,\ \beta>0,\ n\geqslant 2r+\alpha+\beta$, there hold
\begin{align*}
\|wB_{n,r}^{\ast}(f)\|\leqslant C\|wf\|,\ f\in C_{w},\\
\|w(B_{n,r}^{\ast}(f)-f)\|\leqslant \left\{
\begin{array}{lrr}
{\frac{C}{n^{r}}} (\|wf\|+\|w\varphi^{2r}f^{(2r)}\|),  f\in
W_{w}^{2r},    \\
C(\omega_{\varphi}^{2r}(f,n^{-1/2})_{w}+n^{-r}\|wf\|),   f\in C_{w}.
              \end{array}
\right.\\
\|w\varphi^{2r}B_{n,r}^{\ast(2r)}(f)\|\leqslant \left\{
\begin{array}{lrr}
Cn^{r}\|wf\|,    f\in C_{w},    \\
C(\|wf\|+\|w\varphi^{2r}f^{(2r)}\|),    f\in W_{w}^{2r}.
              \end{array}
\right.
\end{align*}
and for $0< \gamma <2r,$
$$\|w(B_{n,r}^{\ast}(f)-f)\|=O(n^{-\gamma/2}) \Longleftrightarrow
\omega_{\varphi}^{2r}(f,t)_{w}=O(t^{r}).$$
Ditzian and Totik \cite{Totik} extended this method of combinations
and defined the following combinations of Bernstein operators:
\begin{align*}
B_{n,r}(f,x):=\sum_{i=0}^{r-1}C_{i}(n)B_{n_i}(f,x).
\end{align*}
with the conditions
\begin{enumerate}[(a)]
\item $n=n_0<n_1< \cdots <n_{r-1}\leqslant
Cn,$\\
\item $\sum_{i=0}^{r-1}|C_{i}(n)|\leqslant C,$\\
\item
$\sum_{i=0}^{r-1}C_{i}(n)=1,$\\
\item $\sum_{i=0}^{r-1}C_{i}(n)n_{i}^{-k}=0$,\ for $k=1,\ldots,r-1$.
\end{enumerate}
\section{The main results}
Now, we can define our \emph{new combinations of Bernstein
operators} as follows:
\begin{align}
B^\ast_{n,r}(f,x):=B_{n,r}(F_n,x)=\sum_{i=0}^{r-1}C_{i}(n)B_{n_i}(F_{n},x),\label{s1}
\end{align}
where $C_{i}(n)$ satisfy the conditions (a)-(d). For the details, it
can be referred to \cite{S. Yu}. Our main results are the following:
\begin{thrm}\label{t1} If \ $\alpha, \ \beta >0,$ for any $f \in
C_w,$ we have
\begin{align}
\|wB^{\ast(r)}_{n,r-1}(f)\| \leqslant Cn^{r}\|wf\|.\label{s2}
\end{align}
\end{thrm}
\begin{thrm}\label{t2} For any $\alpha, \ \beta >0,\ 0
\leqslant \lambda \leqslant 1,$ we have
\begin{align}
|w(x)\varphi^{r\lambda}(x)B^{\ast(r)}_{n,r-1}(f,x)|\leqslant \left\{
\begin{array}{lrr}
Cn^{r/2}{\{max\{n^{r(1-\lambda)/2},\varphi^{r(\lambda-1)}(x)\}\}}\|wf\|,  &&  f\in C_w,    \\
C\|w\varphi^{r\lambda}f^{(r)}\|,&& f\in W_{w,\lambda}^{r}.\label{s3}
              \end{array}
\right.
\end{align}
\end{thrm}
\begin{thrm}\label{t3} For $f\in C_w,\ \alpha, \ \beta>0,\ \alpha_0 \in(0,r), \ 0
\leqslant \lambda \leqslant 1,$ we have
\begin{align}
w(x)|f(x)-B^\ast_{n,r-1}(f,x)|=O((n^{-{\frac
12}}\varphi^{-\lambda}(x)\delta_n(x))^{\alpha_0})
\Longleftrightarrow
\omega_{\varphi^\lambda}^r(f,t)_w=O(t^{\alpha_0}).\label{s4}
\end{align}
\end{thrm}

\section{Lemmas}
\begin{lem}(\cite{Zhou}) For any non-negative real $u$ and $v$, we
have
\begin{align}
\sum_{k=1}^{n-1}({\frac kn})^{-u}(1-{\frac
kn})^{-v}p_{n,k}(x)\leqslant Cx^{-u}(1-x)^{-v}.\label{s5}
\end{align}
\end{lem}
\begin{lem}(\cite{Della}) If $\gamma \in R,$ then
\begin{align}
\sum_{k=0}^n|k-nx|^\gamma p_{n,k}(x) \leqslant Cn^{\frac
\gamma2}\varphi^\gamma(x).\label{s6}
\end{align}
\end{lem}
\begin{lem} For any $f\in W_{w,\lambda}^{r},$ $0 \leqslant \lambda
\leqslant 1$ and $\alpha,\ \beta
>0$, we have
\begin{align}
\|w\varphi^{r\lambda}F_{n}^{(r)}\| \leqslant
C\|w\varphi^{r\lambda}f^{(r)}\|.\label{s7}
\end{align}
\end{lem}
\begin{proof} By symmetry, we only prove the above result when $x\in
(0,1/2]$, the others can be done similarly. Obviously, when $x \in
(0,1/n],$ by \cite{Totik}, we have
\begin{align*}
|L_r^{(r)}(f,x)| \leqslant C|\overrightarrow{\Delta}_{\frac
1r}^{r}f(0)| \ \leqslant \ Cn^{-{\frac r2}+1}\int_0^{\frac rn}u^{\frac r2}|f^{(r)}(u)|du\\
\leqslant Cn^{-{\frac
r2}+1}\|w\varphi^{r\lambda}f^{(r)}\|\int_0^{\frac rn}u^{\frac
r2}w^{-1}(u)\varphi^{-r\lambda}(u)du.
\end{align*}
So
\begin{align*}
|w(x)\varphi^{r\lambda}(x)F_{n}^{(r)}(x)| \leqslant
C\|w\varphi^{r\lambda}f^{(r)}\|.
\end{align*}
If $x\in [{\frac 1n},{\frac 2n}],$ we have
\begin{align*}
|w(x)\varphi^{r\lambda}(x)F_{n}^{(r)}(x)| \leqslant |w(x)\varphi^{r\lambda}(x)f^{(r)}(x)| + |w(x)\varphi^{r\lambda}(x)(f(x)-F_{n}(x))^{(r)}|\\
:= I_1 + I_2.
\end{align*}
For $I_2,$ we have
\begin{align*}
f(x)-F_{n}(x) = (\psi(nx-1)+1)(f(x)-L_r(f,x)).\\
w(x)\varphi^{r\lambda}(x)|(f(x)-F_{n}(x))^{(r)}|=w(x)\varphi^{r\lambda}(x)\sum_{i=0}^rn^i|(f(x)-L_r(f,x))^{(r-i)}|.
\end{align*}
By \cite{Totik}, then
\begin{align*}
|(f(x)-L_r(f,x))^{(r-i)}|_{[{\frac 1n},{\frac 2n}]} \leqslant
C(n^{r-i}\|f-L_r\|_{[{\frac 1n},{\frac 2n}]} +
n^{-i}\|f^{(r)}\|_{[{\frac 1n},{\frac 2n}]}),\ 0<j<r.
\end{align*}
Now, we estimate
\begin{align}
I:=w(x)\varphi^{r\lambda}(x)|f(x)-L_r(x)|.\label{s8}
\end{align}
By Taylor expansion, we have
\begin{align}
f({\frac in})=\sum_{u=0}^{r-1}\frac{({\frac
in}-x)^u}{u!}f^{(u)}(x)+{\frac 1{(r-1)!}}\int_{x}^{{\frac
in}}({\frac in}-s)^{r-1}f^{(r)}(s)ds,\label{s9}
\end{align}
It follows from (\ref{s9}) and the identity
\begin{align*}
\sum\limits_{i=1}^{r}({\frac in})^{v}l_{i}(x)=Cx^v,\ v=0,1,\cdots,r.
\end{align*}
We have
\begin{align*}
L_r(f,x)=\sum_{i=1}^{r}\sum_{u=0}^{r-1}\frac{({\frac
in}-x)^u}{u!}f^{(u)}(x)l_{i}(x)+{\frac
1{(r-1)!}}\sum_{i=1}^{r}l_{i}(x)\int_{x}^{{\frac in}}({\frac in}-s)^{r-1}f^{(r)}(s)ds\nonumber\\
=f(x)+C\sum_{u=1}^{r-1}f^{(u)}(x)(\sum_{v=0}^{u}C_{u}^{v}(-x)^{u-v}\sum_{i=1}^{r}({\frac in})^{v}l_{i}(x))\nonumber\\
+\ \ {\frac 1{(r-1)!}}\sum_{i=1}^{r}l_{i}(x)\int_{x}^{{\frac
in}}({\frac in}-s)^{r-1}f^{(r)}(s)ds,
\end{align*}
which implies that
$${w(x)}\varphi^{r\lambda}(x)|f(x)-L_r(f,x)|={\frac 1{r!}}{w(x)}\varphi^{r\lambda}(x)\sum_{i=1}^{r}l_{i}(x)\int_{x}^{{\frac in}}({\frac in}-s)^{r-1}f^{(r)}(s)ds,$$
since $|l_{i}(x)|\leqslant C$ for $x\in [0,{\frac 2n}],\
i=1,2,\cdots,r$.
It follows from $\frac{|{\frac in}-s|^{r-1}}{w(s)}\leqslant
\frac{|{\frac in}-x|^{r-1}}{w(x)},$ $s$ between ${\frac in}$ and
$x$, then
\begin{align*}
w(x)\varphi^{r\lambda}(x)|f(x)-L_r(f,x)|\leqslant Cw(x)\varphi^{r\lambda}(x)\sum_{i=1}^{r}\int_{x}^{{\frac in}}({\frac in}-s)^{r-1}|f^{(r)}(s)|ds\nonumber\\
\leqslant C\varphi^{r\lambda}(x)\|w\varphi^{r\lambda}f^{(r)}\|\sum_{i=1}^{r}\int_{x}^{{\frac in}}({\frac in}-s)^{r-1}\varphi^{-r\lambda}(s)ds\nonumber\\
\leqslant {\frac {C}{n^r}}\|w\varphi^{r\lambda}f^{(r)}\|.
\end{align*}
Thus
\begin{align*}
I \leqslant C\|w\varphi^{r\lambda}f^{(r)}\|.
\end{align*}
So, we get
\begin{align*}
I_2 \leqslant C\|w\varphi^{r\lambda}f^{(r)}\|.
\end{align*}
Above all, we have
\begin{align*}
|w(x)\varphi^{r\lambda}(x)F_{n}^{(r)}(x)| \leqslant
C\|w\varphi^{r\lambda}f^{(r)}\|.
\end{align*}
\end{proof}
\begin{lem} If $f\in W_{w,\lambda}^{r},$ $0 \leqslant \lambda \leqslant 1$
and $\alpha,\ \beta
>0$, then
\begin{align}
|w(x)(f(x)-L_r(f,x))|_{[0,{\frac 2n}]} \leqslant C(\frac
{\delta_n(x)}{\sqrt{n}\varphi^\lambda(x)})^r\|w\varphi^{r\lambda}f^{(r)}\|.\label{s10}\\
|w(x)(f(x)-R_r(f,x))|_{[1-{\frac 2n},1]} \leqslant C(\frac
{\delta_n(x)}{\sqrt{n}\varphi^\lambda(x)})^r\|w\varphi^{r\lambda}f^{(r)}\|.\label{s11}
\end{align}
\end{lem}
\begin{proof} By Taylor expansion, we have
\begin{align}
f({\frac in})=\sum_{u=0}^{r-1}\frac{({\frac
in}-x)^u}{u!}f^{(u)}(x)+{\frac 1{r!}}\int_{x}^{{\frac in}}({\frac
in}-s)^{r-1}f^{(r)}(s)ds,\label{s12}
\end{align}
It follows from (\ref{s12}) and the identity
\begin{align*}
\sum\limits_{i=1}^{r-1}({\frac in})^{v}l_{i}(x)=Cx^v,\
v=0,1,\ldots,r.
\end{align*}
We have
\begin{align*}
L_r(f,x)=\sum_{i=1}^{r}\sum_{u=0}^{r-1}\frac{({\frac
in}-x)^u}{u!}f^{(u)}(x)l_{i}(x)+{\frac
1{(r-1)!}}\sum_{i=1}^{r}l_{i}(x)\int_{x}^{{\frac in}}({\frac in}-s)^{r-1}f^{(r)}(s)ds\nonumber\\
=f(x)+C\sum_{u=1}^{r-1}f^{(u)}(x)(\sum_{v=0}^{u}C_{u}^{v}(-x)^{u-v}\sum_{i=1}^{r}({\frac in})^{v}l_{i}(x))\nonumber\\
+\ \ {\frac 1{(r-1)!}}\sum_{i=1}^{r}l_{i}(x)\int_{x}^{{\frac
in}}({\frac in}-s)^{r-1}f^{(r)}(s)ds,
\end{align*}
which implies that
$$w(x)|f(x)-L_r(f,x)|={\frac 1{(r-1)!}}w(x)\sum_{i=1}^{r}l_{i}(x)\int_{x}^{{\frac in}}({\frac in}-s)^{r-1}f^{(r)}(s)ds,$$
since $|l_{i}(x)|\leqslant C$ for $x\in [0,{\frac
2n}],\ i=1,2,\cdots,r$.  \\

It follows from $\frac{|{\frac in}-s|^{r-1}}{w(s)}\leqslant
\frac{|{\frac in}-x|^{r-1}}{w(x)},$ $s$ between ${\frac in}$ and
$x$, then
\begin{align*}
w(x)|f(x)-L_r(f,x)|\leqslant Cw(x)\sum_{i=1}^{r}\int_{x}^{{\frac in}}({\frac in}-s)^{r-1}|f^{(r)}(s)|ds\nonumber\\
\leqslant C{\frac
{\varphi^r(x)}{\varphi^{r\lambda}(x)}}\|w\varphi^{r\lambda}f^{(r)}\|\sum_{i=1}^{r}\int_{x}^{{\frac in}}({\frac in}-s)^{r-1}\varphi^{-r}(s)ds\nonumber\\
\leqslant C{\frac
{\delta_n^r(x)}{\varphi^{r\lambda}(x)}}\|w\varphi^{r\lambda}f^{(r)}\|\sum_{i=1}^{r}\int_{x}^{{\frac in}}({\frac in}-s)^{r-1}\varphi^{-r}(s)ds\nonumber\\
\leqslant C(\frac
{\delta_n(x)}{\sqrt{n}\varphi^\lambda(x)})^r\|w\varphi^{r\lambda}f^{(r)}\|.
\end{align*}

The proof of (\ref{s11}) can be done similarly.
\end{proof}
\begin{lem}(\cite{S. Yu}) For every $\alpha,\ \beta>0,$ we have
\begin{align}
\|wB^\ast_{n,r-1}(f)\| \leqslant C\|wf\|.\label{s13}
\end{align}
\end{lem}
\begin{lem}(\cite{Wang}) If \
$\varphi(x)=\sqrt{x(1-x)},\ 0 \leqslant \lambda \leqslant 1,\ 0
\leqslant \beta \leqslant 1,$ then
\begin{align}
\int_{-{\frac {h\varphi^\lambda(x)}{2}}}^{\frac
{h\varphi^\lambda(x)}{2}} \cdots \int_{-{\frac
{h\varphi^\lambda(x)}{2}}}^{\frac
{h\varphi^\lambda(x)}{2}}\varphi^{-r\beta}(x+\sum_{k=1}^ru_k)du_1
\cdots du_r \leqslant Ch^r\varphi^{r(\lambda-\beta)}(x).\label{s14}
\end{align}
\end{lem}
\section{Proof of Theorems}
\subsection{Proof of Theorem \ref{t1}}

By symmetry, in what follows we will always assume that $x \in
(0,{\frac 12}].$ It is sufficient to prove the validity for
$B_{n,r-1}(F_n,x)$ instead of $B^\ast_{n,r-1}(f,x).$ When $x\in
(0,{\frac {1}{n}}),$ we have
\begin{align*}
|w(x)B_{n,r-1}^{\ast(r)}(f,x)|\leqslant w(x)\sum_{i=0}^{r-2}{\frac
{n_{i}!}{({n_{i}-r})!}}\sum_{k=0}^{n_i-r}|\overrightarrow{\Delta}_{\frac
1{n_i}}^{r}F_{n}{(\frac k{n_i})}|p_{n_i-r,k}(x)\nonumber\\
\leqslant
Cw(x)\sum_{i=0}^{r-2}n_{i}^{r}\sum_{k=0}^{n_i-r}|\overrightarrow{\Delta}_{\frac
1{n_i}}^{r}F_{n}{(\frac k{n_i})}|p_{n_i-r,k}(x)\nonumber\\
\leqslant
Cw(x)\sum_{i=0}^{r-2}n_{i}^{r}\sum_{k=0}^{n_i-r}\sum_{j=0}^{r}C_{r}^{j}|F_{n}({\frac
{k+r-j}{n_i}})|p_{n_i-r,k}(x)\nonumber\\
\leqslant
Cw(x)\sum_{i=0}^{r-2}n_{i}^{r}\sum_{j=0}^{r}C_{r}^{j}|F_{n}({\frac
{r-j}{n_i}})|p_{n_i-r,0}(x)\nonumber\\
+ \
Cw(x)\sum_{i=0}^{r-2}n_{i}^{r}\sum_{j=0}^{r}C_{r}^{j}|F_{n}({\frac
{n_{i}-j}{n_i}})|p_{n_i-r,n_i-r}(x)\nonumber\\
+ \
Cw(x)\sum_{i=0}^{r-2}n_{i}^{r}\sum_{k=1}^{n_i-r-1}\sum_{j=0}^{r}C_{r}^{j}|F_{n}({\frac
{k+r-j}{n_i}})|p_{n_i-r,k}(x)\nonumber\\
:=H_1 +H_2 + H_3.
\end{align*}

We have
\begin{align*}
H_1\leqslant
Cw(x)\|wf\|\sum_{i=0}^{r-2}n_{i}^{r}w^{-1}(\frac {1}{n_i})p_{n_i-r,0}(x)\\
\leqslant C\|wf\|\sum_{i=0}^{r-2}n_{i}^{r}(n_ix)^{\alpha}(1-x)^{n_i-r}\\
\leqslant Cn^{r}\|wf\|.
\end{align*}

When $1 \leqslant k \leqslant n_i-r-1,$ we have $1 \leqslant k+2r-j
\leqslant n_i-1,$ and thus
\begin{align*}
{\frac {w({\frac {k}{n_i-r}})}{w({\frac {k+r-j}{n_i}})}}=({\frac
{n_i}{n_i-r}})^{\alpha+\beta}({\frac {k}{k+r-j}})^\alpha({\frac
{n_i-r-k}{n_i-r-k+j}})^\beta \leqslant C.
\end{align*}
Thereof, by (\ref{s5}), we have
\begin{align*}
H_3 \leqslant
Cw(x)\|wF_n\|\sum_{i=0}^{r-2}n_{i}^{r}\sum_{k=1}^{n_i-r-1}\sum_{j=0}^{r}{\frac
{1}{w({\frac {k+r-j}{n_i}})}}p_{n_i-r,k}(x)\\
\leqslant
Cw(x)\|wF_n\|\sum_{i=0}^{r-2}n_{i}^{r}\sum_{k=1}^{n_i-r-1}{\frac
{1}{w({\frac {k}{n_i-r}})}}p_{n_i-r,k}(x)\\
\leqslant Cn^r\|wF_n\| \ \leqslant \ Cn^r\|wf\|.
\end{align*}

Similarly, we can get $H_2\leqslant Cn^{r}\|wf\|$. So
\begin{align}
|w(x)B^{\ast(r)}_{n,r-1}(f,x)| \leqslant Cn^{r}\|wf\|,\ x\in
(0,{\frac {1}{n}}).\label{s15}
\end{align}
When $x\in [{\frac {1}{n}},{\frac 12}],$ according to \cite{Totik},
we have
\begin{align*}
|w(x)B_{n,r-1}^{\ast(r)}(f,x)|\\
=|w(x)B_{n,r-1}^{(r)}({F_{n}},x)|\\
\leqslant w(x)(\varphi^{2}(x))^{-r}\sum_{i=0}^{r-2}\sum_{j=0}^{r}|Q_{j}(x,n_i)|n_{i}^{j}\sum_{k=0}^n|(x-{\frac kn_{i}})^{j}F_{n}({\frac kn_{i}})|p_{n_i,k}(x).\\
\end{align*}
Then\\
$Q_{j}(x,n_i)=(n_ix(1-x))^{[{\frac {r-j}{2}}]},$ and
$(\varphi^{2}(x))^{-r}Q_{j}(x,n_i)n_{i}^{j}\leqslant
C(n_{i}/\varphi^{2}(x))^{\frac {r+j}{2}},$ we have
\begin{align}
|w(x)B_{n,r-1}^{\ast(r)}(f,x)|\leqslant
Cw(x)\sum_{i=0}^{r-2}\sum_{j=0}^{r}(\frac
{n_{i}}{\varphi^{2}(x)})^{\frac
{r+j}{2}}\sum_{k=0}^{n_{i}}|(x-{\frac
kn_{i}})^{j}F_{n}({\frac kn_{i}})|p_{n_i,k}(x)\nonumber\\
\leqslant C\|wF_n\|w(x)\sum_{i=0}^{r-2}\sum_{j=0}^{r}(\frac
{n_{i}}{\varphi^{2}(x)})^{\frac {r+j}{2}}\sum_{k=0}^{n_i}{\frac
{|x-{\frac {k}{n_{i}}|^{j}}}{w({\frac
{k^\ast}{n_i})}}}p_{n_i,k}(x),\label{s16}
\end{align}
where $k^\ast =1$ for $k=0,$ $k^\ast =n_i-1$ for $k=n_i$ and $k^\ast
=k$ for $1<k<n_i.$ Note that
\begin{align*}
w^2(x){\frac {p_{n_i,0}(x)}{w^2({\frac {1}{n_{i}}})}} \leqslant
C(n_ix)^{2\alpha}(1-x)^{n_i} \leqslant C,
\end{align*}
and
\begin{align*}
w^2(x){\frac {p_{n_i,n_i}(x)}{w^2(1-{\frac {1}{n_{i}}})}} \leqslant
Cn_i^\beta x^{n_i} \leqslant C{\frac {n_i^\beta}{2^{n_i}}} \leqslant
C.
\end{align*}
By (\ref{s5}), we have
\begin{align}
\sum_{k=0}^{n_i}{\frac {1}{w^2({\frac
{k^\ast}{n_{i}}})}}p_{n_i,k}(x) \leqslant Cw^{-2}(x).\label{s17}
\end{align}
Now, applying Cauchy's inequality, by (\ref{s6}) and (\ref{s17}), we
have
\begin{align*}
\sum_{k=0}^{n_i}{\frac {{|x-{\frac {k}{n_{i}}|^{j}}}}{{w({\frac
{k^\ast}{n_i}})}}}p_{n_i,k}(x) \leqslant
({\sum_{k=0}^{n_i}{|x-{\frac
{k}{n_{i}}|^{2j}}}p_{n_i,k}(x)})^{1/2}({\sum_{k=0}^{n_i}{\frac
{1}{{w^2({\frac {k^\ast}{n_i}})}}}p_{n_i,k}(x)})^{1/2}\\
\leqslant Cn_i^{-j/2}\varphi^j(x)w^{-1}(x).
\end{align*}
Substituting this to (\ref{s16}), we have
\begin{align}
|w(x)B^{\ast(r)}_{n,r-1}(f,x)| \leqslant Cn^{r}\|wf\|,\ x\in [{\frac
{1}{n}},{\frac 12}].\label{s18}
\end{align}
We get Theorem \ref{t1} by (\ref{s15}) and (\ref{s18}). $\Box$
\subsection{Proof of Theorem \ref{t2}}

(1) When $f\in C_w$, we proceed it as follows:\\~\\
\textit{Case 1.} If $0\leqslant \varphi(x)\leqslant {\frac
{1}{\sqrt{n}}}$, by $(\ref{s2})$, we have
\begin{align}
|w(x)\varphi^{r\lambda}(x)B_{n,r-1}^{\ast(r)}(f,x)|\leqslant
Cn^{-r\lambda/2}|w(x)B_{n,r-1}^{\ast(r)}(f,x)|\leqslant
Cn^{r(1-\lambda/2)}\|wf\|.\label{s19}
\end{align}
\textit{Case 2.} If $\varphi(x)> {\frac {1}{\sqrt{n}}}$, we have
\begin{align*}
|B_{n,r-1}^{\ast(r)}(f,x)|=|B_{n,r-1}^{(r)}(F_{n},x)|\nonumber\\
\leqslant(\varphi^{2}(x))^{-r}\sum_{i=0}^{r-2}\sum_{j=0}^{r}|Q_{j}(x,n_i)C_{i}(n)|n_{i}^{j}\sum_{k=0}^{n_i}|(x-{\frac
kn_{i}})^{j}F_{n}({\frac kn_{i}})|p_{n_i,k}(x),
\end{align*}
$Q_{j}(x,n_i)=(n_ix(1-x))^{[{\frac {r-j}{2}}]},$ and
$(\varphi^{2}(x))^{-r}Q_{j}(x,n_i)n_{i}^{j}\leqslant
C(n_i/\varphi^{2}(x))^{\frac {r+j}{2}}$.\\ So
\begin{align}
|w(x)\varphi^{r\lambda}(x)B_{n,r-1}^{\ast(r)}(f,x)|\nonumber\\
\leqslant
Cw(x)\varphi^{r\lambda}(x)\sum_{i=0}^{r-2}\sum_{j=0}^{r}({\frac
{n_{i}}{\varphi^2(x)}})^{\frac {r+j}{2}}\sum_{k=0}^{n_i}|(x-{\frac
kn_{i}})^{j}F_{n}({\frac kn_{i}})|p_{n_i,k}(x)\nonumber\\
\leqslant Cn^{\frac r2}\varphi^{r(\lambda-1)}(x).\label{s20}
\end{align}
It follows from combining with (\ref{s19}) and (\ref{s20}) that the first inequality is proved.\\
(2) When $f\in W_{w,\lambda}^{r},$ we have
\begin{align}
B_{n,r-1}^{(r)}(F_{n},x)=\sum_{i=0}^{r-2}C_{i}(n)n_{i}^{r}\sum_{k=0}^{n_{i}-r}\overrightarrow{\Delta}_{\frac
1{n_{i}}}^{r}F_{n}({\frac kn_{i}})p_{n_i-r,k}(x).\label{s21}
\end{align}
If $0<k<n_{i}-r,$ we have
\begin{align}
|\overrightarrow{\Delta}_{\frac 1{n_{i}}}^{r}F_{n}({\frac kn_{i}})
|\leqslant Cn_{i}^{-r+1}\int_{0}^{\frac
{r}{n_{i}}}|F_{n}^{(r)}({\frac kn_{i}}+u)|du,\label{s22}
\end{align}
If $k=0,$ we have
\begin{align}
|\overrightarrow{\Delta}_{\frac 1{n_{i}}}^{r}F_{n}(0)|\leqslant
C\int_{0}^{\frac {r}{n_{i}}}u^{r-1}|F_{n}^{(r)}(u)|du,\label{s23}
\end{align}
Similarly
\begin{align}
|\overrightarrow{\Delta}_{\frac 1{n_{i}}}^{r}F_{n}({\frac
{n_{i}-r}{n_{i}}})|\leqslant Cn_i^{-r+1}\int_{1-{\frac
{r}{n_{i}}}}^{1}(1-u)^{\frac r2}|F_{n}^{(r)}(u)|du.\label{s24}
\end{align}
By (\ref{s21})-(\ref{s24}), we have
\begin{align}
|w(x)\varphi^{r\lambda}(x)B_{n,r-1}^{\ast(r)}(f,x)|\nonumber\\\leqslant
C{w(x)}\varphi^{r\lambda}(x)\|w\varphi^{r\lambda}F_n^{(r)}\|\sum_{i=0}^{r-2}\sum_{k=0}^{n_{i}-r}(w\varphi^{r\lambda})^{-1}(\frac
{k^\ast}{n_i})p_{n_i-r,k}(x),\label{s25}
\end{align}
where $k^\ast =1$ for $k=0,$ $k^\ast =n_i-r-1$ for $k=n_i-r$ and
$k^\ast =k$ for $1<k<n_i-r.$ By (\ref{s5}), we have
\begin{align}
\sum_{k=0}^{n_i-r}(w\varphi^{r\lambda})^{-1}(\frac
{k^\ast}{n_i})p_{n_i-r,k}(x) \leqslant
C(w\varphi^{r\lambda})^{-1}(x).\label{s26}
\end{align}
which combining with (\ref{s26}) give
\begin{align*}
|w(x)\varphi^{r\lambda}(x)B_{n,r-1}^{\ast(r)}(f,x)|\leqslant
C\|w\varphi^{r\lambda}f^{(r)}\|.\Box
\end{align*}
So we get the second inequality of the Theorem \ref{t2}.
\subsection{Proof of Theorem \ref{t3}}
\subsubsection{The direct theorem} We know
\begin{align}
F_n(t)=F_n(x)+F'_n(t)(t-x) + \cdots + {\frac{1}{(r-1)!}}\int_x^t
(t-u)^{r-1}F_n^{(r)}(u)du,\label{s27}\\
B_{n,r-1}((\cdot-x)^k,x)=0, \ k=1,2,\cdots,r-1.\label{s28}
\end{align}
According to the definition of $W_{w,\lambda}^{r},$ \ for any $g \in
W_{w,\lambda}^{r},$ we have
$B^\ast_{n,r-1}(g,x)=B_{n,r-1}(G_{n}(g),x),$ and
$w(x)|G_{n}(x)-B_{n,r-1}(G_{n},x)|=w(x)|B_{n,r-1}(R_r(G_n,t,x),x)|,$
thereof $R_r(G_n,t,x)=\int_x^t (t-u)^{r-1}G^{(r)}_n(u)du.$ It
follows from ${\frac {|t-u|^{r-1}}{w(u)}} \leqslant {\frac
{|t-x|^{r-1}}{w(x)}},$ $u$ between $t$ and $x$, we have
\begin{align}
w(x)|G_{n}(x)-B_{n,r-1}(G_{n},x)|  \leqslant
C\|w\varphi^{r\lambda}G^{(r)}_n\|w(x)B_{n,r-1}({\int_x^t{\frac
{|t-u|^{r-1}}{w(u)\varphi^{r\lambda}(u)}du,x}})\nonumber\\
 \leqslant
C\|w\varphi^{r\lambda}G^{(r)}_n\|w(x)(B_{n,r-1}(\int_x^t{\frac
{|t-u|^{r-1}}{\varphi^{2r\lambda}(u)}}du,x))^{\frac 12}\cdot
\nonumber\\
(B_{n,r-1}(\int_x^t{\frac {|t-u|^{r-1}}{w^2(u)}}du,x))^{\frac
12}.\label{s29}
\end{align}
also
\begin{align}
\int_x^t{\frac {|t-u|^{r-1}}{\varphi^{2r\lambda}(u)}}du \leqslant
C{\frac {|t-x|^r}{\varphi^{2r\lambda}(x)}},\ \int_x^t{\frac
{|t-u|^{r-1}}{w^2(u)}}du \leqslant {\frac
{|t-x|^r}{w^2(x)}}.\label{s30}
\end{align}
By (\ref{s6}), (\ref{s29}) and (\ref{s30}), we have
\begin{align}
w(x)|G_{n}(x)-B_{n,r-1}(G_{n},x)| \leqslant
C\|w\varphi^{r\lambda}G^{(r)}_n\|\varphi^{-r\lambda}(x)B_{n,r-1}
(|t-x|^r,x)\nonumber\\
\leqslant Cn^{-\frac r2}{\frac
{\varphi^{r}(x)}{\varphi^{r\lambda}(x)}}\|w\varphi^{r\lambda}G^{(r)}_n\|\nonumber\\
\leqslant Cn^{-\frac r2}{\frac
{\delta_n^r(x)}{\varphi^{r\lambda}(x)}}\|w\varphi^{r\lambda}G^{(r)}_n\|\nonumber\\
= C(\frac
{\delta_n(x)}{\sqrt{n}\varphi^\lambda(x)})^r\|w\varphi^{r\lambda}G^{(r)}_n\|.\label{s31}
\end{align}
By (\ref{s7}), (\ref{s10}), (\ref{s11}) and (\ref{s31}), when $g \in
W_{w,\lambda}^{r},$ then
\begin{align}
w(x)|g(x)-B^\ast_{n,r-1}(g,x)| \leqslant w(x)|g(x)-G_{n}(g,x)| +
w(x)|G_{n}(g,x)-B^\ast_{n,r-1}(g,x)|\nonumber\\
\leqslant |w(x)(g(x)-L_r(g,x))|_{[0,{\frac 2n}]} +
|w(x)(g(x)-R_r(g,x))|_{[1-{\frac 2n},1]}\nonumber\\  +\ \ C(\frac
{\delta_n(x)}{\sqrt{n}\varphi^\lambda(x)})^r\|w\varphi^{r\lambda}G^{(r)}_n\|\nonumber\\
\leqslant C(\frac
{\delta_n(x)}{\sqrt{n}\varphi^\lambda(x)})^r\|w\varphi^{r\lambda}g^{(r)}\|.\label{s32}
\end{align}
For $f \in C_w,$ we choose proper $g \in W_{ w,\lambda}^{r},$ by
(\ref{s13}) and (\ref{s32}), then
\begin{align*}
w(x)|f(x)-{B^\ast_{n,r-1}(f,x)}| \leqslant w(x)|f(x)-g(x)| +
w(x)|{B^\ast_{n,r-1}(f-g,x)}| \\+
w(x)|g(x)-{B^\ast_{n,r-1}(g,x)}|\\
\leqslant C(\|w(f-g)\|+(\frac
{\delta_n(x)}{\sqrt{n}\varphi^\lambda(x)})^r\|w\varphi^{r\lambda}g^{(r)}\|)\\
\leqslant C\omega_{\varphi^\lambda}^r(f,\frac
{\delta_n(x)}{\sqrt{n}\varphi^\lambda(x)})_w. \Box
\end{align*}
\subsubsection{The inverse theorem} We define the weighted main-part
modulus for $D=R_+$ by
\begin{align*}
\Omega_{\varphi^\lambda}^r(C,f,t)_w = \sup_{0<h \leqslant
t}\|w\Delta_{{h\varphi}^\lambda}^rf\|_{[Ch^\ast,\infty]},\\
\Omega_{\varphi^\lambda}^r(1,f,t)_w =
\Omega_{\varphi^\lambda}^r(f,t)_w.
\end{align*}
where $C>2^{1/\beta(0)-1},\ \beta(0)>0$ and $h^\ast$ is given by
\begin{align*}
h^\ast= \left\{
\begin{array}{lrr}
(Ar)^{1/1-\beta(0)}h^{1/1-\beta(0)}, &&  0 \leqslant \beta(0) <1,
    \\
0, &&   \beta(0) \geqslant 1.
              \end{array}
\right.
\end{align*}
The main-part $K$-functional is given by
\begin{align*}
K_{r,\varphi^\lambda}(f,t^r)_w=\sup_{0<h \leqslant
t}\inf_g\{\|w(f-g)\|_{[Ch^\ast,\infty]}+t^r\|w\varphi^{r\lambda}g^{(r)}\|_{[Ch^\ast,\infty]},
\end{align*}
where $g^{(r-1)} \in A.C.((Ch^\ast,\infty))\}$. By (\cite{Totik}),
we have
\begin{align}
C^{-1}\Omega_{\varphi^\lambda}^r(f,t)_w \leqslant
\omega_{\varphi^\lambda}^{r}(f,t)_w \leqslant
C\int_0^t{\frac {\Omega_{\varphi^\lambda}^r(f,\tau)_w}{\tau}}d\tau,\label{s33} \\
C^{-1}K_{r,\varphi^\lambda}(f,t^r)_w \leqslant
\Omega_{\varphi^\lambda}^r(f,t)_w \leqslant
CK_{r,\varphi^\lambda}(f,t^r)_w.\label{s34}
\end{align}
\begin{proof} Let $\delta>0,$ by (\ref{s34}), we choose proper $g$ so
that
\begin{align}
\|w(f-g)\| \leqslant C\Omega_{\varphi^\lambda}^r(f,\delta)_w,\
\|w\varphi^{r\lambda}g^{(r)}\| \leqslant
C\delta^{-r}\Omega_{\varphi^\lambda}^r(f,\delta)_w.\label{s35}
\end{align}
then
\begin{align}
|w(x)\Delta_{h\varphi^\lambda}^rf(x)| \leqslant
|w(x)\Delta_{h\varphi^\lambda}^r(f(x)-{B^\ast_{n,r-1}(f,x)})|+|w(x)\Delta_{h\varphi^\lambda}^rB^\ast_{n,r-1}(f-g,x)|\nonumber\\
 +\ |w(x)\Delta_{h\varphi^\lambda}^r{B^\ast_{n,r-1}(g,x)}|\nonumber\\
\leqslant \sum_{j=0}^rC_r^j(n^{-\frac
12}{\frac {\delta_n(x+({\frac r2}-j)h\varphi^\lambda(x))}{\varphi^\lambda(x+({\frac r2}-j)h\varphi^\lambda(x))}})^{\alpha_0}\nonumber\\
 +\ \ \int_{-{\frac {h\varphi^\lambda(x)}{2}}}^{\frac
{h\varphi^\lambda(x)}{2}}\cdots \int_{-{\frac
{h\varphi^\lambda(x)}{2}}}^{\frac
{h\varphi^\lambda(x)}{2}}w(x){B^{\ast(r)}_{n,r-1}(f-g,x+\sum_{k=1}^ru_k)}du_1\cdots
du_r\nonumber\\
 +\ \ \int_{-{\frac {h\varphi^\lambda(x)}{2}}}^{\frac
{h\varphi^\lambda(x)}{2}}\cdots \int_{-{\frac
{h\varphi^\lambda(x)}{2}}}^{\frac
{h\varphi^\lambda(x)}{2}}w(x){B^{\ast(r)}_{n,r-1}(g,x+\sum_{k=1}^ru_k)}du_1\cdots
du_r\nonumber\\
:= J_1+J_2+J_3.\label{s36}
\end{align}
Obviously
\begin{align}
J_1 \leqslant C(n^{-\frac 12}\varphi^{-\lambda}(x)\delta_n(x))^{\alpha_0}.\label{s37}
\end{align}
By (\ref{s2}) and (\ref{s35}), we have
\begin{align}
J_2 \leqslant Cn^r\|w(f-g)\|\int_{-{\frac
{h\varphi^\lambda(x)}{2}}}^{\frac {h\varphi^\lambda(x)}{2}}\cdots
\int_{-{\frac {h\varphi^\lambda(x)}{2}}}^{\frac
{h\varphi^\lambda(x)}{2}}du_1 \cdots du_r\nonumber\\
\leqslant Cn^rh^r\varphi^{r\lambda}(x)\|w(f-g)\|\nonumber\\
\leqslant
Cn^rh^r\varphi^{r\lambda}(x)\Omega_{\varphi^\lambda}^r(f,\delta)_w.\label{s38}
\end{align}
By the first inequality of (\ref{s3}), we let $\lambda=1,$ and
(\ref{s14}) as well as (\ref{s35}), we have
\begin{align}
J_2 \leqslant Cn^{\frac r2}\|w(f-g)\|\int_{-{\frac
{h\varphi^\lambda(x)}{2}}}^{\frac {h\varphi^\lambda(x)}{2}} \cdots
\int_{-{\frac {h\varphi^\lambda(x)}{2}}}^{\frac
{h\varphi^\lambda(x)}{2}}\varphi^{-r}(x+\sum_{k=1}^ru_k)du_1 \cdots du_r\nonumber\\
\leqslant Cn^{\frac r2}h^r\varphi^{r(\lambda-1)}(x)\|w(f-g)\|\nonumber\\
\leqslant Cn^{\frac
r2}h^r\varphi^{r(\lambda-1)}(x)\Omega_{\varphi^\lambda}^r(f,\delta)_w.\label{s39}
\end{align}
By the second inequality of (\ref{s3}) and (\ref{s35}), we have
\begin{align}
J_3 \leqslant C\|w\varphi^{r\lambda}g^{(r)}\|w(x)\int_{-{\frac
{h\varphi^\lambda(x)}{2}}}^{\frac {h\varphi^\lambda(x)}{2}} \cdots
\int_{-{\frac {h\varphi^\lambda(x)}{2}}}^{\frac
{h\varphi^\lambda(x)}{2}}{w^{-1}(x+\sum_{k=1}^ru_k)}\varphi^{-r\lambda}(x+\sum_{k=1}^ru_k)du_1 \cdots du_r\nonumber\\
\leqslant Ch^r\|w\varphi^{r\lambda}g^{(r)}\|\nonumber\\
\leqslant
Ch^r\delta^{-r}\Omega_{\varphi^\lambda}^r(f,\delta)_w.\label{s40}
\end{align}
Now, by (\ref{s36})-(\ref{s40}), we get
\begin{align*}
|w(x)\Delta_{h\varphi^\lambda}^rf(x)| \leqslant C\{(n^{-\frac
12}\delta_n(x))^{\alpha_0} + h^r(n^{-\frac
12}\delta_n(x))^{-r}\Omega_{\varphi^\lambda}^r(f,\delta)_w +
h^r\delta^{-r}\Omega_{\varphi^\lambda}^r(f,\delta)_w\}.
\end{align*}
When $n \geqslant 2,$ we have
\begin{align*}
n^{-\frac 12}\delta_n(x) < (n-1)^{-\frac 12}\delta_{n-1}(x)
\leqslant \sqrt{2}n^{-\frac 12}\delta_n(x),
\end{align*}
Choosing proper $x, n \in N,$ so that
\begin{align*}
n^{-\frac 12}\delta_n(x) \leqslant \delta < (n-1)^{-\frac
12}\delta_{n-1}(x),
\end{align*}
Therefore
\begin{align*}
|w(x)\Delta_{h\varphi^\lambda}^rf(x)| \leqslant C\{\delta^{\alpha_0}
+ h^r\delta^{-r}\Omega_{\varphi^\lambda}^r(f,\delta)_w\}.
\end{align*}
By Borens-Lorentz lemma in \cite{Totik}, we get
\begin{align}
\Omega_{\varphi^\lambda}^r(f,t)_w \leqslant
Ct^{\alpha_0}.\label{s41}
\end{align}
So, by (\ref{s41}), we get
\begin{align*}
\omega_{\varphi^\lambda}^{r}(f,t)_w \leqslant C\int_0^t{\frac
{\Omega_{\varphi^\lambda}^r(f,\tau)_w}{\tau}}d\tau =
C\int_0^t\tau^{\alpha_0-1}d\tau = Ct^{\alpha_0}.
\end{align*}
\end{proof}

\end{document}